\documentclass{amsart}[12pt]
\usepackage{amsmath, amsthm, amscd, amsfonts,}
\usepackage{graphicx,float}

\usepackage{amssymb}

\DeclareMathOperator{\KH}{KH}
\DeclareMathOperator{\GCH}{GCH}

\DeclareMathOperator{\cf}{cf}
\DeclareMathOperator{\Col}{Col}

\def\MPB{{\mathbb{P}}}

\newtheorem{theorem}{Theorem}[section]

\newtheorem{remark}[theorem]{Remark}

\newtheorem{question}[theorem]{Question}
\numberwithin{equation}{section}

\usepackage{pb-diagram}

\def\rmark{\mbox{$\rm\bf\rule{0.06em}{1.45ex}\kern-0.05em R$}}
\def\pmark{\mbox{$\rm\bf\rule{0.06em}{1.45ex}\kern-0.05em P$}}
\def\nmark{\mbox{$\rm\bf\rule{0.06em}{1.45ex}\kern-0.05em N$}}
\def\vdash{\mbox{$\rm\| \kern-0.13em -$}}

\def\rmark{\mbox{$\rm\bf\rule{0.06em}{1.45ex}\kern-0.05em R$}}
\def\pmark{\mbox{$\rm\bf\rule{0.06em}{1.45ex}\kern-0.05em P$}}
\def\nmark{\mbox{$\rm\bf\rule{0.06em}{1.45ex}\kern-0.05em N$}}
\def\vdash{\mbox{$\rm\| \kern-0.13em -$}}
\newcommand{\lusim}[1]{\smash{\underset{\raisebox{1.2pt}[0cm][0cm]{$\sim$}}
{{#1}}}}

\begin{document}

\title[The generalized Kurepa hypothesis at singular cardinals]{The generalized Kurepa hypothesis at singular cardinals}

\author[M. Golshani ]{Mohammad
  Golshani }


~
~~
~
\thanks{The  author's research has been supported by a grant from IPM (No. 97030417). He also thanks Menachem Magidor
for sharing his ideas during the  14th International Workshop in Set Theory at Luminy, in particular, the results of this paper are obtained from discussions with him.} 

\thanks{In personal communication, Stevo Todorcevic informed the author that  Theorem \ref{main theorem from chang conjecture} has been obtained by him before; for example can be found  on
page 231 of his book ``Walks on ordinals and their characteristics''. However our proof is different from him.}
\maketitle

\begin{abstract}
We discuss the generalized Kurepa hypothesis $\KH_{\lambda}$ at singular cardinals $\lambda$. In particular, we answer  questions of
Erd\H{o}s-Hajnal \cite{erdos-hajnal} and
 Todorcevic \cite{todorcevic1}, \cite{todorcevic2} by showing that $\GCH$ does not imply $\KH_{\aleph_\omega}$ nor the existence of a family $ \mathcal{F} \subseteq [\aleph_\omega]^{\aleph_0}$
of size $\aleph_{\omega+1}$ such that $\mathcal{F} \restriction X$ has size $\aleph_0$ for every $X \subseteq S, |X|=\aleph_0$.
\end{abstract}

\section{introduction}
For an infinite cardinal $\lambda$ let the generalized Kurepa hypothesis at $\lambda,$ denoted $\KH_{\lambda}$, be the assertion: there exists a family $\mathcal{F} \subseteq P(\lambda)$ such that $|\mathcal{F}|> \lambda$ but $|\mathcal{F} \restriction X| \leq |X|$ for every infinite $X \subseteq \lambda, |X| < \lambda$,
where $\mathcal{F} \restriction X=\{ t \cap X: t \in \mathcal{F}  \}$.

By a theorem of Erd\H{o}s-Hajnal-Milner \cite{erdos-hajnal-milner}, if $\lambda$ is a singular cardinal of uncountable cofinality,
$\theta^{\cf(\lambda)} < \lambda$ for all $\theta<\lambda$ and if $\mathcal{F} \subseteq P(\lambda)$ is such that the set $\{\alpha < \lambda:|\mathcal{F} \restriction \alpha| \leq |\alpha|  \}$ is stationary in $\lambda$, then $|\mathcal{F}| \leq \lambda$. In particular, $\GCH$ implies   $\KH_{\lambda}$ fails for all singular cardinals $\lambda$ of uncountable cofinality. On the other hand, by an unpublished result of Prikry \cite{prikry}, $\KH_{\lambda}$ holds in $L,$ the G\"{o}del's constructible universe, for singular cardinals of countable cofinality (see \cite{todorcevic2}).
Later, Todorcevic \cite{todorcevic1}, \cite{todorcevic2} improved Prikry's theorem by showing that if $\lambda$ is a singular cardinal of countable cofinality, then $\Box_\lambda$ implies $\KH_\lambda.$ The following question is asked in \cite{todorcevic1} and \cite{todorcevic2}.

\begin{question}
Does $\GCH$ imply $\KH_{\aleph_\omega}?$
\end{question}
The question is also related to the following question of Erd\H{o}s-Hajnal \cite{erdos-hajnal} (question 19$/$E)
\begin{question}
Assume $\GCH$. Let $|S|=\aleph_\omega$. Does there exist a family $\mathcal{F}, |\mathcal{F}|=\aleph_{\omega+1}, \mathcal{F} \subseteq [S]^{\aleph_0}$
such that $\mathcal{F} \restriction X$ has size $\aleph_0$ for every $X \subseteq S, |X|=\aleph_0?$
\end{question}
We show that, relative to the existence of large cardinals,  both of the above questions can consistently be false, and so they are independent of $\text{ZFC}.$

\section{$\KH_\lambda$ fails above a supercompact cardinal}
In this section we prove the following.
\begin{theorem} \label{above supercompact}
Suppose $\kappa$ is a supercompact cardinal and $\lambda \geq \kappa.$ Then $\KH_\lambda$ fails.
\end{theorem}
\begin{proof}
 Let $\mathcal{F} \subseteq P(\lambda)$ be of size $\geq \lambda^+.$  Let $j: V \to M$ be a $\lambda^+$-supercompactness embedding with
$crit(j)=\kappa$. Also let $U$ be the normal measure on $P_\kappa(\lambda)$ derived from $j$, .i.e.,
\[
U=\{X \subseteq P_\kappa(\lambda): j[\lambda] \in j(X)   \}.
\]
We have
\begin{itemize}
\item $M \models$``$j(\mathcal{F}) \subseteq P(j(\lambda))$ is of size $\geq j(\lambda)^+$''.
\item $j''[\lambda] \in M$ and $M \models$``$|j''[\lambda]|=\lambda < j(\lambda)$''.
\item $\mathcal{F} \in M$
\item $a \neq b \in \mathcal{F} \implies j(a) \cap j''[\lambda] \neq j(b) \cap j''[\lambda].$
\end{itemize}
In particular,
\begin{center}
 $M \models$``$|j(\mathcal{F})\restriction j''[\lambda]| \geq |\mathcal{F}| \geq \lambda^+$''.
 \end{center}
 This implies that
 \[
 \{x \in P_\kappa(\lambda): |\mathcal{F} \restriction x| \geq |x|^+            \} \in U.
 \]
In particular, $\mathcal{F}$ is not a $\KH_\lambda$-family.
\end{proof}
\begin{remark}
The above result is optimal in the sense that we can not in general find a set $x \subseteq \lambda$ of size in the interval $[\kappa, \lambda)$
such that $|\mathcal{F} \restriction x| \geq |x|^+.$ To see this assume $\kappa$
is supercompact and Laver indestructible. Then one can easily define a $\kappa$-directed closed forcing notion which adds a family
$\mathcal{F} \subseteq P(\lambda)$ such that $|\mathcal{F}| \geq \lambda^+,$
but  $|\mathcal{F} \restriction x| \leq |x|$ for any set $x$ with $\kappa \leq |x| < \lambda.$
\end{remark}

\section{The Chang's conjecture and $\KH_{\aleph_\omega}$}
In this section we prove our main theorem by showing a consistent negative answer to the questions of Erd\H{o}s-Hajnal and Todorcevic.
Recall from \cite{levinski} that ``$\GCH+$ the  Chang's conjecture $(\aleph_{\omega+1}, \aleph_\omega) \twoheadrightarrow (\aleph_1, \aleph_0)$''
is consistent,  relative to the existence of a 2-huge cardinal. See also \cite{hayut}, where the large cardinal assumption is reduced to the existence of a $(+\omega+1)$-subcompact cardinal $\kappa$.
\begin{theorem} 
\label{main theorem from chang conjecture}
Assume $\GCH+$Chang's conjecture $(\aleph_{\omega+1}, \aleph_\omega) \twoheadrightarrow (\aleph_1, \aleph_0)$. Then $\KH_{\aleph_\omega}$ fails.
Also, there does not exist a family $\mathcal{F}  \subseteq [\aleph_\omega]^{\aleph_0}, |\mathcal{F}| \geq \aleph_{\omega+1}$
such that $\mathcal{F} \restriction X$ has size $\aleph_0$ for every $X \subseteq S, |X|=\aleph_0.$
\end{theorem}
\begin{proof}
Suppose towards contradiction that there exists a family $\mathcal{F}$ which witnesses $\KH_{\aleph_\omega}$.
Fix a bijection $f: H_{\aleph_{\omega+1}}\leftrightarrow \mathcal{F}$. Consider the structure
\[
\mathcal{A}=(H_{\aleph_{\omega+1}}, \in, \mathcal{F}, \aleph_\omega, f).
\]
Let $\mathcal{B}=(B, \in, \mathcal{G}, A, g   ) \prec \mathcal{A}$ be such that $|B|=\aleph_1$
and $|A|=\aleph_0$.

Note that $\mathcal{A} \models$``$\forall t \in \mathcal{F}, t \subseteq \aleph_\omega,$ and hence
$\mathcal{B} \models$``$\forall t \in \mathcal{G}, t \subseteq A,$
in particular,
$\mathcal{G} \subseteq \mathcal{F} \restriction A.$ On the other hand $g: B \leftrightarrow \mathcal{G}$ is a bijection, hence we have
\[
|\mathcal{F} \restriction A| \geq |\mathcal{G}| = |B|=\aleph_1 > \aleph_0.
\]
We get a contradiction and the result follows.

Similar argument shows that there can not be a family $\mathcal{F} \subseteq [\aleph_\omega]^{\aleph_0}$
as stated above.
\end{proof}

Mohammad Golshani,
School of Mathematics, Institute for Research in Fundamental Sciences (IPM), P.O. Box:
19395-5746, Tehran-Iran.

E-mail address: golshani.m@gmail.com

URL: http://math.ipm.ac.ir/golshani/

\end{document}